\documentclass[12pt,oneside,A4, reqno]{amsart}

\usepackage{amsmath,amsthm,amsfonts,mathabx,tikz,tikz-cd,stmaryrd,url,mathtools, braket}
\usepackage[margin=1in]{geometry}
\usepackage{hyperref}
\hypersetup{
	colorlinks,
	citecolor=black,
	filecolor=black,
	linkcolor=black,
	urlcolor=black
}

\usetikzlibrary{matrix,arrows,decorations.pathmorphing}

\theoremstyle{definition}
\newtheorem{defn}{Definition}[section]
\theoremstyle{proposition}
\newtheorem{prop}{Proposition}[section]
\theoremstyle{lemma}
\newtheorem{lemma}{Lemma}[section]
\theoremstyle{counterexample}
\newtheorem{cex}{Counterexample}[section]

\begin{document}

\title{Note on filtered colimits of Hilbert spaces}
\author{Branko Nikoli\'c and Alessandra Di Pierro}
\address{Uviversity of Verona, Verona, Italy}
\email{branik.mg@gmail.com, alessandra.dipierro@univr.it}

\date{\today}



\thanks{The first author gratefully acknowledges the financial support through the Cooperint grant.}

\begin{abstract}
The category of Hilbert spaces and contractions has filtered colimits, and tensoring preserves them. We also discuss (problems with) bounded maps.
\end{abstract}

\maketitle


\section{Introduction}

While working on formally adding certain directed colimits to the fusion category of Fibonacci anyons, we noticed a gap in the literature in the formal examination of directed colimits in a more basic case -- different categories of Hilbert spaces.

Let $\mathbf{Hilb}_\mathrm{b}$, $\mathbf{Hilb}_\mathrm{c}$ and $\mathbf{Hilb}_\mathrm{r}$  denote the categories Hilbert spaces and bounded (respectively contracting, isometric) linear operators. While $\mathbf{Hilb}_\mathrm{b}$ has direct sums, $\mathbf{Hilb}_\mathrm{c}$ has $\aleph_1$-filtered colimits \cite[Proposition 3.4.2]{Makkai1989}, and $\mathbf{Hilb}_\mathrm{r}$ has ($\aleph_0$- or finitely-) filtered colimits \cite{Karvonen2019, Lieberman2019}. The statement of our main result, given in Proposition \ref{prop:main}, seems to be contained in \cite[Lemma 5.3]{Heunen2008}, but their proof uses \cite[Example 2.3.9]{Adamek1994}, which, like \cite[Proposition 3.4.2]{Makkai1989}, is about $\aleph_1-$filtered colimits. 

In Section \ref{sec:prelim} we briefly review different choices for morphisms of Hilbert spaces, and types of filtered/directed colimits. In Section \ref{sec:Hc} we present construction of  colimits of chains of contractions, which guarantees existence of arbitrary filtered colimits \cite{Adamek1994}. In Section \ref{sec:Hb} we take a look at what can be done for bounded morphisms, and what causes problems.

\section{Prelimineries}\label{sec:prelim}
The braket $\braket{x|y}$ denotes the inner product of $x$ and $y$. The norm $\vert x\vert$ of a vector $x$ equals $\sqrt{\braket{x|x}}$. 
\begin{defn}
A linear operator $G:H_1\rightarrow H_2$ between Hilbert spaces is:
\begin{itemize}
\item \textit{bounded} if $\vert Gx\vert\leq b\vert x\vert$ for some $b$ and all $x$ -- the norm of $G$, denoted $\vert G\vert$ is the smallest such $b$,
\item a \textit{contraction} if $|G|\leq 1$,
\item an \textit{isometry} if $\braket{Gx|Gx}=\braket{x|x}$, in which case $\vert G\vert=1.$
\end{itemize}
\end{defn}
Hilbert spaces, with these arrows form categories $\mathbf{Hilb}_\mathrm{b}$, $\mathbf{Hilb}_\mathrm{c}$ and $\mathbf{Hilb}_\mathrm{r}$. There are identity-on-objects inclusion $\mathbf{Hilb}_\mathrm{r}\hookrightarrow\mathbf{Hilb}_\mathrm{c}\hookrightarrow\mathbf{Hilb}_\mathrm{b}$.

\begin{defn}
A \textit{$\lambda$-directed poset} is a poset $P$ in which every subset of cardinality smaller than $\lambda$ has an upper bound. A category $\mathcal{C}$ has \textit{$\lambda$-directed colimits} if for all $\lambda$-directed posets $P$, and all functors $F:P\rightarrow \mathcal{C}$, the colimit of $F$ exists.
\end{defn}

\begin{defn}\cite[Remark 1.21]{Adamek1994}
A \textit{$\lambda$-filtered category} is a category $\mathcal{D}$ in which every subcategory, with cardinality (of arrows) smaller than $\lambda$, has a cocone. A category $\mathcal{C}$ has \textit{$\lambda$-filtered colimits} if for all $\lambda$-filtered categories $\mathcal{D}$, and all functors $F:\mathcal{D}\rightarrow \mathcal{C}$, the colimit of $F$ exists.
\end{defn}
Every filtered diagram has a directed subdiagram whose colimit, if exists, guarantees existence, and coincides with the colimit of the original diagram \cite[Theorem 1.5]{Adamek1994}.

Note that increasing $\lambda$ makes conditions on $P$ stronger, reduces the number of diagrams required to have a colimit, and increases the number of categories that are $\lambda$-directed cocomplete. Dually, decreasing $\lambda$, weakens conditions on $P$, broadens the class of required colimit diagrams, and decreases the number of categories that are $\lambda$-directed cocomplete.

In our case, a $\aleph_1$-directed diagram requires specification of an upper bound for any subset of size $\aleph_0$ (countable subset), which is not very useful if, for example, one wants to consider an $\omega$-chain of embeddings -- it would require to specify the colimit of the diagram in the diagram.

\section{Filtered colimits in $\mathbf{Hilb_c}$}\label{sec:Hc}

\begin{lemma}\label{lemma:innerBound}
If $G:H_1\rightarrow H_2$ is a linear contraction then
\begin{equation*}
\vert\braket{x|y} - \braket{Gx|Gy}\vert^2\leq
(
\braket{x|x} - \braket{Gx|Gx}
) \cdot
(
\braket{y|y} - \braket{Gy|Gy}
).
\end{equation*}
\end{lemma}
\begin{proof}
The definition of an adjoint operator gives $\braket{Gx|Gy}=\braket{x|G^\dagger Gy}$. For the Hermitian (self-adjoint) operator $G^\dagger G$ we can choose a (generalised\footnote{For simplicity, this proof is written assuming that $G^\dagger G$ is compact and $H_1$ is separable. In the case of a continuum spectrum, replace sums by integrals; in the case of inseparable $H_1$ consider a separable subspace invariant under $G^\dagger G$ that contains $x$ and $y$.}) basis that diagonalises it, with (generalised) eigenvalues $0\leq \lambda_i\leq 1$. Expansion of $x$ and $y$ in the same basis gives
\begin{align*}
\braket{x|y}&=\sum_i \bar{x}_i y_i\\
\braket{Gx|Gy}&=\sum_i \lambda_i\bar{x}_i y_i\\
\braket{x|y}-\braket{Gx|Gy}&=\sum_i (1-\lambda_i)\bar{x}_i y_i\\
&= \sum_i (\sqrt{1-\lambda_i}\bar{x}_i) (\sqrt{1-\lambda_i}y_i).
\end{align*}
Using the Cauchy-Schwartz inequality on the last two lines we obtain
\begin{align*}
\vert\braket{x|y}-\braket{Gx|Gy}\vert^2&
\leq \sum_i (1-\lambda_i)\bar{x}_i x_i
\sum_i (1-\lambda_i)\bar{y}_i y_i\\
&= \left(\braket{x|x}-\braket{Gx|Gx} \right)
\left(\braket{y|y}-\braket{Gy|Gy} \right).
\end{align*}  
\end{proof}
\begin{prop}\label{prop:main}
The category of Hilbert spaces and linear contractions has ($\aleph_0$-)filtered colimits.
\end{prop}
\begin{proof}
Every filtered diagram has a directed subdiagram whose colimit coincides with the colimit of the original diagram \cite[Theorem 1.5]{Adamek1994}. A category with colimits of chains (ordinals) has directed colimits \cite[Corollary 1.7]{Adamek1994}. Furthermore, a chain-colimit-preserving functor between such categories also preserves direceted and filtered colimits. Hence, we will restrict our analysis to colimits of chains.

For a chain diagram $$F:\mathcal{D}\rightarrow \mathbf{Hilb_c}$$
with $A$, $B$, etc. and $f$, $g$, etc. denoting objects and arrows of $\mathcal{D}$. Following
\cite[Proposition 3.4.2]{Makkai1989}, form the underlying colimit in $\mathrm{Vect}$. Explicitly, vectors of $L=\mathrm{colim} F$ are equivalence classes on disjoint union of underlying sets of $FA$, formed by relating $$FA\ni x\sim Ffx\in FB,$$ where $f:A\rightarrow B$, and closing the relation to form an equivalence relation. Denote by $(A,x)$ the equivalence class of $x\in FA$. Linear combination of $(A,x)$ and $(C,y)$ is formed by linearly combining representatives in the space $FD$, where $D$ is an upper bound (or maximum) of $A$ and $C$. 

For the inner product of $(A,x)$ and $(C,y)$, consider the image set 
\begin{equation*}
N=\{(|Fax|,|Fcy|)\}_{A\xrightarrow{a} D \xleftarrow{c} C}\subset \mathbb{R}^2
\end{equation*}
The cardinality of the set $N$ is at most continuum, regardless of the cardinality, or structure, of the chain (ordinal) $\mathcal{D}$. $N$ is lower-bounded by $0$ and decreasing in both coordinates, so there is a limit point $l=(n_x,n_y)$, where $n_x=\inf(\pi_x N)$ (and $n_y=\inf(\pi_y N)$), is the infimum of the set of norms of images of $x$ (respectively $y$). Choose a (decreasing in both variables) Cauchy sequence $l_n\subset N$ whose limit is $l$ and any $\omega$-chain 
\begin{equation*}
D_0\xrightarrow{d_0}D_1
\xrightarrow{d_1}D_2\ldots
\end{equation*}
in $\mathcal{D}$ such that $l_n=(|x_n|,|y_n|)$, with $(D_n,x_n)$ and $(D_n,y_n)$ being the unique relabelings of respective classes of $(A,x)$ and $(C,y)$. Using the fact that that $\braket{x_n|x_n}$ and $\braket{y_n|y_n}$ are Cauchy sequences, and Lemma \ref{lemma:innerBound}, it follows that $\braket{x_n|y_n}$ is a Cauchy sequence, with the converging value being the definition of the inner product $\braket{(A,x)|(C,y)}$. The value of the inner product is independent of which Cauchy sequence and chain are chosen -- for any two Cauchy sequences, and corresponding chains, the sorted union Cauchy sequence and the corresponding chain must have the same limit value as both of its forming subsequences. Linearity and conjugate symmetry follow component-wise, while positive-definiteness may be lost -- it is necessary to formally equate classes of zero distance. Metric completion, using the inner product, produces the colimiting Hilbert space. 

Colimit inclusions $\iota_A$, mapping $x\in FA$ to $(A,x)$, are linear contractions. Given a cocone, that is, components $\alpha_A:FA\rightarrow H$, satisfying
\begin{equation}\label{eq:cocone}
\alpha_A=\alpha_B\circ Ff,
\end{equation}
form a mapping $\alpha:L\rightarrow H$ by defining $\alpha(A,x)=\alpha_A(x)$ -- two representatives of the same class map to the same vector in $H$ either because of (\ref{eq:cocone}) if they initially belonged to the same class, or because $\alpha_{D_n}(x-y)=0$, if classes $(A,x)$ and $(C,y)$ are identified after noting that distance between them is $0$.  Mapping on Cauchy sequences is extended by continuity. Since each $\alpha_A$ is a contraction, we have $\braket{\alpha(A,x)|\alpha(A,x)}\leq\braket{x_n|x_n}$ for any $D_n$ in the chain defining $\braket{(A,x)|(A,x)}$, so $\alpha$ is a contraction. Take $\beta :L\rightarrow H$ to be any map such that $\beta\circ\iota_A=\alpha_A$ -- after applying both sides to $x\in FA$ we conclude that $\beta(A,x)=\alpha(A,x)$. Uniqueness of the continuous extension to the whole completion gives that $\beta=\alpha$.
\end{proof}

\begin{prop}
The tensor product in $\mathbf{Hilb_c}$ preserves directed colimits - there is a unitary isomorphism 
\begin{equation*}
c:\mathrm{colim}_i H \otimes K_i
\rightarrow H\otimes \mathrm{colim}_i K_i
\end{equation*}
natural in $H$.
\end{prop}
\begin{proof}
We use the same notation as in the previous proof. The mapping $c$ is defined on colimit components by
$$c_i:H\otimes K_i\xrightarrow{1\otimes\iota_i} 
H\otimes \mathrm{colim}_i K_i,$$
where $\iota_i$ is the colimit inclusion of $K_i$.

To see $c$ is an isometry, note that the inner product of $(A,h\otimes x)$ and $(C,h'\otimes y)$ is $\braket{h,h'}\cdot\mathrm{lim}_n\braket{x_n|y_n}$, which is the same as the inner product after applying $c$ to both vectors. This extends by linearity and completion to the whole of $\mathrm{colim}_i H \otimes K_i.$

Finally, the image of $c$ is dense in
$H\otimes \mathrm{colim}_i K_i$ -- take a finite sum $s=\sum_i h_i\otimes l_i$, where
$l_i = \mathrm{lim}_{j}(A_{i,j},x_{i,j})$ is a Cauchy sequence limit, then $(A_{i,j},h_i\otimes x_{i,j})$ is a Cauchy sequence in $\mathrm{colim}_i H \otimes K_i$, with limit $p_i$, and $c(p_i)=h_i\otimes l_i$ due to continuity of $c$. Hence, $s=\sum_i c(p_i)$, the image of $c$ is dense, and the adjoint, $c^\dagger$, is also an isometry. 

Naturality follows from the universal property of the colimit.
\end{proof}

\section{Normalisation of $\omega$-chains in $\textbf{Hilb}_\mathrm{b}$}\label{sec:Hb}

Note that normalising bounded linear maps $b\mapsto b/\vert b\vert$ is not a functor, since $\vert b'\circ b\vert\leq\vert b'\vert\cdot \vert b\vert$. However, normalising $\omega$-chains works as follows.
\begin{prop}
There is a normalisation endofunctor $N$ on $\mathbf{Cat}(\omega,\mathbf{Hilb}_\mathrm{b})$, and a natural isomorphism $1\cong N$.
\end{prop}
\begin{proof}
Define $r(b)>0$ to be any function on operators greater or equal than the operator norm. For example, $r(0)=1$ and $r(b)=\vert b\vert$ otherwise. Or, if continuity is preferable, $r(b)=1$ if b is a contraction and $r(b)=\vert b\vert$ otherwise.

A chain $C$ and its normalisation $N_rC$ are given by the two rows of the the following diagram
\begin{equation*}
\begin{tikzpicture}[scale=1.5]
\node (A1) at (0,1)
{$C_0$};
\node (A2) at (1,1)
{$C_1$};
\node (B1) at (0,0)
{$C_0$};
\node (B2) at (1,0)
{$C_1$};

\def \s {3};
\node (dotA) at ( 0.5 * \s + 0.5, 1) {\ldots};
\node (dotB) at ( 0.5 * \s + 0.5, 0) {\ldots};
\node (dot2A) at ( 1.5 * \s + 0.5, 1) {\ldots};
\node (dot2B) at ( 1.5 * \s + 0.5, 0) {\ldots};
\node (A3) at (\s+0,1)
{$C_n$};
\node (A4) at (\s+1,1)
{$C_{n+1}$};
\node (B3) at (\s+0,0)
{$C_n$};
\node (B4) at (\s+1,0)
{$C_{n+1}$};

\path[->,font=\scriptsize,>=angle 90]
(A1) edge node[above] {$c_0$} (A2)
(B1) edge node[below] {$\frac{c_0}{r(c_0)}$} (B2)
(A1) edge node[left] {$1$} (B1)
(A2) edge node[right] 
	{$\frac{1}{r(c_0)}$} (B2)
(A2) edge (dotA) 
(dotA) edge (A3)
(A4) edge (dot2A)
(B2) edge (dotB) 
(dotB) edge (B3)
(B4) edge (dot2B)
(A3) edge node[above] {$c_n$} (A4)
(B3) edge node[below] {$\frac{c_n}{r(c_n)}$} (B4)
(A3) edge[->] node[left] {$\prod_{i=0}^{n-1}\frac{1}{r(c_i)}$} (B3)
(A4) edge node[right] 
	{$\prod_{i=0}^{n}\frac{1}{r(c_i)}$} (B4);
\end{tikzpicture}
\end{equation*}
while the isomorphism $\eta_C$ of chains is given on chain components by vertical arrows.

A morphism of chains $\alpha:C\rightarrow D$, given by components $\alpha_n:C_n\rightarrow D_n$, is mapped to
$$(N_r\alpha)_n=\alpha_n\prod_{i=0}^{n-1}\frac{r(c_i)}{r(d_i)}.$$
$N_r\alpha$ is indeed a chain morphism.
$N_r$ respects the identity and composition because it does componentwise. Finally, $\eta$ is natural because
\begin{equation*}
\begin{tikzpicture}[xscale=3,yscale=1.5]
\node (A3) at (0,1)
{$C_n$};
\node (A4) at (1,1)
{$D_{n}$};
\node (B3) at (0,0)
{$C_n$};
\node (B4) at (1,0)
{$D_{n}$};

\path[->,font=\scriptsize,>=angle 90]
(A3) edge node[above] {$\alpha_n$} (A4)
(B3) edge node[below] {$\alpha_n\prod_{i=0}^{n-1}\frac{r(c_i)}{r(d_i)}$} (B4)
(A3) edge[->] node[left] {$\prod_{i=0}^{n-1}\frac{1}{r(c_i)}$} (B3)
(A4) edge node[right] 
	{$\prod_{i=0}^{n}\frac{1}{r(d_i)}$} (B4);
\end{tikzpicture}
\end{equation*}
commutes.
\end{proof}
However, the colimit of directed contractions from Proposition \ref{prop:main} does not have the required universal property with respect to bounded maps.
\begin{cex}
Let $H_n=\mathbb{C}$, and $e_n:H_{n}\rightarrow H_{n+1}$ defined by scaling by $1/2$. Then the colimit of the omega chain is the zero space. However, scaling by $2^n$ produces a (non-zero) cocone $b_n:H_n\rightarrow \mathbb{C}$.
\end{cex}
Similar argument is true if we restrict to colimits of diagrams of isometries.
\begin{cex}
Let $H_n=\mathbb{C}^n$, and $e_n:H_{n}\hookrightarrow H_{n+1}$ embedding of the first $n$ components. Then the colimit of the omega chain is the infinite separable Hilbert space $H_\omega$. There is a cocone $b_n:H_n\rightarrow H_\omega$ defined by scaling the $n^\mathrm{th}$ component by $n$. The induced linear map is not bounded.
\end{cex}

Both of the counterexamples have components whose norm increases without bound as we go to infinity. A straightforward generalisation of the argument in Proposition \ref{prop:main} leads to the following statement. 
\begin{prop}
Let $L$ be a directed colimit of $F$ in $\mathbf{Hilb}_\mathrm{c}$, and $\alpha:F\Rightarrow H$ a cocone in $\mathbf{Hilb}_\mathrm{b}$. The uniquely induced bounded map $L\rightarrow H$ exists if and only if the cocone $\alpha$ is globaly bounded, that is, there is $b$, such that $\vert \alpha_A\vert\leq b$ for all $\alpha_A$ in the cocone. If all colimit inclusions $\iota_A$ have trivial kernels (non-zero vectors map to non-zero vectors), then the induced linear map exists, but has no bound.
\end{prop}

\bibliographystyle{acm}

\end{document}